\documentclass[a4paper,10pt, oneside]{article}
\usepackage{amsmath, amssymb, amsthm}
\usepackage[titletoc, title]{appendix}
\usepackage{mathrsfs}
\usepackage{paralist}
\usepackage{graphics} 
\usepackage{epsfig} 
\usepackage{multirow}
\usepackage{graphicx}
\usepackage{epstopdf}
\graphicspath{{fig/}}
\usepackage[pdftex, pdfpagelabels, bookmarks, hyperindex, hyperfigures,
colorlinks, pagebackref]{hyperref}
\usepackage[capitalise, nosort]{cleveref}

\crefname{equation}{}{}
\crefname{lem}{Lemma}{Lemmas}
\crefname{thm}{Theorem}{Theorems}

\newcommand{\snmii}[1]
{
  \left\vert\kern-0.25ex
  \left\vert\kern-0.25ex
  \left\vert
  #1
  \right\vert\kern-0.25ex
  \right\vert\kern-0.25ex
  \right\vert
}

\newtheorem{assumption}{Assumption}

\newtheorem{lemma}{Lemma}[section]
\newtheorem{remark}{Remark}[section]
\newtheorem{theorem}{Theorem}[section]
\newtheorem{example}{Example}[section]

\numberwithin{equation}{section}

\title{A Nitsche-eXtended finite element method 
for distributed optimal control problems of elliptic interface equations
\thanks
{
	This work was supported by  National Natural Science Foundation of China (11771312).
}}
\author{
	 \ Tao Wang  \thanks{Email: wangtao5233@hotmail.com }, Chaochao Yang \thanks{Email: yangchaochao9055@163.com}, \
	Xiaoping Xie \thanks{Corresponding author. Email: xpxie@scu.edu.cn} \\
	{School of Mathematics, Sichuan University, Chengdu 610064, China}
}

\begin{document}
	\maketitle

	\begin{abstract}
	
	This paper analyzes an  interface-unfitted numerical method for  distributed optimal control problems governed by elliptic interface equations. We follow the variational
discretization concept to discretize the optimal control problems, and apply a  Nitsche-eXtended  finite element method 
to discretize the corresponding state and adjoint equations, where piecewise cut basis functions around the interface
are enriched into the standard linear  element space. 
Optimal error estimates of the state, co-state and control in a mesh-dependent norm and the $L^2$ norm are derived.  Numerical results are provided to verify the theoretical results.

	\end{abstract}	

	\section{Introduction}
	Optimization processes in multi-physics progress or engineering design with different materials usually lead to 
	optimal control problems
governed by  partial differential  equations with interfaces. 
%
%
	In this paper, we consider the following distributed  optimal control problem: 
	\begin{equation}\label{eqobjective}
	\text{min}~ J(y, u):=\frac{1}{2}\int_\Omega (y-y_d)^2~dx+\frac{\nu}{2}\int_\Omega u^2~dx
	\end{equation}
	for $ (y,u)\in H^1_0(\Omega)\times L^2 (\Omega) $ subject to the elliptic interface problem 
	\begin{equation}\label{eqstrongstate0}
	\left\{
	\begin{array}{rll}
	& -\nabla\cdot(\alpha(x)\nabla y)=f+u, & \text{ in }\Omega \\
	& y=0, & \text{ on }\partial\Omega \\
	& [y]=0,[\alpha\nabla_n y ]=g, & \text{ on }\Gamma\\
	\end{array}
	\right.
	\end{equation}
	with the control constraint
	\begin{equation}\label{eqconstraint}
	u_0\leq u \leq u_1, \text{ a.e. in } \Omega.
	\end{equation}
Here 
$\Omega\subseteq\mathbb{R}^d (d=2,3)$ is  a polygonal or polyhedral domain, consisting of two disjoint subdomains $\Omega_i (1\leq i\leq 2)$, and interface $\Gamma=\partial\Omega_1\cap \partial\Omega_2$.
 $y_d\in L^2(\Omega)$ is the desired state to be achieved by controlling  $u$, and $\nu$ is a positive constant. $\alpha(x)$ is   piecewise   constant with
  $\alpha|_{\Omega_i}=\alpha_i>0$ for $i=1,2$,
$[y]:=(y|_{\Omega_1})|_\Gamma-(y|_{\Omega_2})|_\Gamma$  is  the jump  of function $y$ across interface $\Gamma$,  $\textbf{n}$ is the 
  unit  
 normal vector along $\Gamma$ pointing to $\Omega_1$, $\nabla_\textbf{n} y=n\cdot\nabla y$  is the normal derivative of $y$, $f \in L^2(\Omega)$, 
$g\in L^{2} (\Gamma)$, 
 and  $u_0, u_1 \in L^2 (\Omega)$ with $u_0\leq u_1$ a.e. in $\Omega$. The choice of homogeneous boundary condition on boundary $\partial\Omega$ is made for ease of presentation, since similar results are valid for other boundary conditions.
 	
	For an elliptic interface problem, it is well-known that its solution 
	is generally not in $H^2(\Omega)$  due to the discontinuity of coefficient. This low regularity may lead to   reduced accuracy for numerical approximations $\cite{Babu1970The,Xu1982Estimate}$.  In literature there are usually two types of methods to improve the numerical accuracy,  interface(or body)-fitted methods \cite{Barretts87Fit, Brambles96fin, Chen98, Plums03opt, Lis10opt, Cai17dis} and interface-unfitted methods.  For the interface-fitted methods, meshes aligned with the interface are used so as to dominate the approximation error caused by the non-smoothness of solution. 
	 However, it is often difficult or expensive to generate complicated interface-fitted meshes, especially when the interface is moving with time or iteration. 
	
	In contrast with  the interface-fitted methods, the interface-fitted methods, with   certain types of  modification for approximating functions around the interface, can avoid using the interface-fitted meshes. One typical type of interface-unfitted methods is the extended/generalized finite element method (\quad XFEM/GFEM) (cf. \cite{Babu1994Special,Strouboulis2000The,I2012Stable,moes1999a,Nicaise2011Optimal}),   where additional basis functions characterizing the singularity of solution around the interface are enriched into the corresponding approximation space. We 
refer to 
\cite{Kergrene2016Stable,Soghrati2012An,Cheng2010Higher} for the numerical simulation of XFEM/GFEM for some elliptic interface problems.
The immersed finite element method (IFEM) 
(cf. \cite{Camps06qua, Li06imm, Lins07err,Yan2007Immersed}) is another typical type of  interface-unfitted methods, where   special finite element basis  functions  are constructed to satisfy the homogeneous interface jump conditions in a certain sense.  We note   that it is usually not  easy to extend the IFEM to the case of non-homogeneous interface conditions \cite{Yan2007Immersed,He2011Immersed, Han2016A} and,  as pointed out in    \cite{Lin2015Partially}, the classic IFEM may lead to deteriorate accuracy, while partially penalized  IFEMs,  with extra stabilization terms introduced
at interface edges for penalizing the discontinuity in IFE functions,  are   optimally
convergent. 
	
	In  \cite{Hansbo2002An}, a special XFEM with  optimal convergence  was proposed for the elliptic interface problems. This method, called  Nitsche-XFEM,  combines the idea of XFEM  with Nitsche's approach  \cite{Nitsche71ube}, where
	   additional cut basis functions which are discontinuous across the interface are added into   the standard linear finite element space,    and  the   parameters in the Nitsche's numerical fluxes on each  element intersected by the interface are chosen to depend on the relative area/volume of the two parts aside the interface.  For the development of interface-unfitted methods using additional cut basis functions, we refer to   \cite{ Becker2009A, Barrau2012A,Lehrenfeld2012Analysis, Hansbo14cut, Burman15cut, Wang2016}. 
	
For optimal control problems governed by elliptic equations  with smooth coefficients,   a lot of work on finite element methods can be found in literature; see \cite{Hinze05var,Chens10err,Yang2018,Wengs16sta,Yang17} for control constraints, see \cite{Lius09new, Hintermuller10goa,Roschs17rel} for state constraints, see \cite{Beckers00ada,Gongs16ada,Lis02ada,Schneiders16pos} for adaptive convergence analysis. 
However, there are only  limited papers on the numerical analysis   for  optimal control problems of elliptic interface equations.   In \cite{Zhang2015Immersed}    the classic IFEM was  applied to discretize the model \eqref{eqobjective}-\eqref{eqconstraint} with the homogeneous interface jump condition $g=0$.  
In \cite{wachsmuth2016optimal},   $hp$-finite elements were investigated for  the optimal control problems of  elliptic interface equations on interface-fitted meshes.
In a very recent work \cite{Yang2018},   the Nitsche-XFEM  was applied for interface optimal control problems of  elliptic interface equations and shown to have  optimal convergence. 
	
	
%

In this paper, we shall follow the  variational discretization concept and  apply the Nitsche-XFEM   for the numerical solution of the  distributed  optimal control problem \eqref{eqobjective}-\eqref{eqconstraint}.  
%
 Optimal error estimates will be derived for the state, co-state and control on  meshes  independent of the interface.  
	
	The remainder of the paper is organized as following. Section 2 introduces some notations and the  optimality conditions for the optimal control problem.  Section 3 gives a brief introduction for Nitsche-XFEM and several theoretical results associated with this method. In section 4, we discretize the optimal control problem, show its discrete optimality conditions, and derive error estimates  for the state, co-state and control of the optimal control problem. Section 5  describes an iteration algorithm for the discrete system, and Section 6 provides several numerical examples to verify our theoretical results.  Finally,  Section 7 gives concluding remarks.

	\section{Notation and optimality conditions}
	For any bounded domain $\Lambda \subset  \mathbb{R}^d$ and non-negative integer $j$, let $H^{j}(\Lambda)$ and $H^{j}_0(\Lambda)$  denote the standard    Sobolev spaces on $\Lambda$ with    norm $\|\cdot\|_{j, \Lambda}$ and semi-norm $|\cdot|_{j,\Lambda}$. In particular, $L^2(\Lambda):=H^0(\Lambda)$, with  the standard $L^2$-inner product $(\cdot,\cdot)_\Lambda$. When $\Lambda=\Omega$, we use abbreviations $\|\cdot\|_{j}:=\|\cdot\|_{j, \Omega}$,  $|\cdot|_{j}:=|\cdot|_{j, \Omega}$, and $(\cdot,\cdot):=(\cdot,\cdot)_\Omega$.  
We also need the fractional Sobolev space
$$H^{j+\frac{1}{2}} (\Lambda) := \{w\in H^j (\Lambda): \sum_{|\alpha|=j}\iint_{\Lambda\times\Lambda} \frac{|D^\alpha w(s)-D^\alpha w(t)|^2}{|s-t|^{d+1}}~dsdt <\infty\}$$
with  norm
$$\|w\|_{j+\frac{1}{2},\Lambda} :=\left(\|w\|_{j,\Lambda}^2+\sum_{|\alpha|=j} \iint_{\Lambda\times\Lambda} \frac{|D^\alpha w(s)-D^\alpha w(t)|^2}{|s-t|^{d+1}}~dsdt \right)^{\frac12}.$$
For $s\in \mathbb{R}^+$, we define
$$ H^s(\Omega_1\cup\Omega_2):=\left\{w\in L^2(\Omega):\ w|_{\Omega_i} \in H^s(\Omega_i),\ i=1,2   \right\}$$
with norm $$\|\cdot\|_{s,\Omega_1\cup\Omega_2}:=\left(\sum\limits_{i=1}^2 \|\cdot\|_{s,\Omega_i}^2 \right)^{\frac{1}{2}}.$$

Throughout this paper, we use $\bar{a}\lesssim \bar{b}$ to denote $\bar{a}\leq C\bar{b}$, where  $C$ is  a generic positive constant $C$  independent of the mesh parameter  $h$  and the location of the interface relative to the corresponding mesh. 

%
The weak formulation of state equation \eqref{eqstrongstate} is as follows:  find $y \in H^1_0(\Omega)$ such that
\begin{equation}
\label{eqweakform}
a(y,v)=(u+f,v)+(g,v)_{\Gamma}, \quad \forall v \in H^1_0(\Omega),
\end{equation} 
where $a(y,v):=(\alpha \nabla y ,\nabla v)$.
It is easy to see that problem \eqref{eqweakform} admits   a unique solution. We  make the following regularity assumptions for the solution $y$.
\begin{assumption} \label{A1}It holds $y\in H_0^1 (\Omega)\cap H^{3/2} (\Omega_1 \cup \Omega_2)$ and
\begin{equation}
\label{regu1}
\|y\|_{\frac{3}{2},\Omega_1 \cup \Omega_2} \lesssim  \|u\|_0+\|f\|_0+\|g\|_{0,\Gamma}.
\end{equation}
In addition, if $g\in H^{1/2} (\Gamma)$, then   $y\in H_0^1 (\Omega)\cap H^2 (\Omega_1 \cup \Omega_2)$ and
\begin{equation}
\label{regu2}
\|y\|_{2,\Omega_1 \cup \Omega_2} \lesssim  \|u\|_0+\|f\|_0+\|g\|_{\frac12,\Gamma}.
\end{equation}
\end{assumption}

\begin{remark} \label{assumption-regularity}

We note  that the Assumption 1 is reasonable. In fact,  if   $\Omega$ and   $\Gamma$ are smooth with $\Gamma\cap \partial\Omega=\emptyset$, then the regularity \eqref{regu1} holds \cite[(2.2)]{Brambles96fin}.  And it has been shown in  \cite[Corollary 4.12]{wachsmuth2016optimal} that \eqref{regu1} holds if   $\Omega \subset \mathbb{R}^2$ and its subdomains $\Omega_i$ are all polygonal. As for the regularity \eqref{regu2}, if the domain $\Omega$ is convex, and the interface $\Gamma$ is $C^2$ continuous with $\Gamma\cap \partial\Omega=\emptyset$, then \eqref{regu2}   holds \cite[theorem 2.1]{Chen98}.
\end{remark}


	
	Define $$U_{ad}:=\{ u\in L^2(\Omega):u_{a}\leq u \leq u_{b} \ \text{a.e. in} \ \Omega \}.$$ By using the standard technique in $\cite{Tr2010Optimal}$, we can easily derive the optimality conditions for the optimal control problem \eqref{eqobjective}-\eqref{eqconstraint}.
	\begin{lemma} \label{leoptimalcondition}
		The optimal control problem  \eqref{eqobjective}-\eqref{eqconstraint} has a unique solution $(y,u)\in H_0^1(\Omega)\times U_{ad}$, and the equivalent optimality conditions read: the triple $(y,p,u)\in H_0^1(\Omega)\times H_0^1(\Omega)\times U_{ad}$ satisfies 
\begin{align}
& a(y,v)=(u+f,v)+(g,v)_{\Gamma},~~\forall v\in H_0^1(\Omega),\label{eqstate}\\
& a(v,p)=(y-y_d,v),~~\forall v\in H_0^1(\Omega),\label{eqcostate}\\
& (p+a u,v-u) \geq 0,~~\forall v\in U_{ad}.\label{eqprojection}
\end{align}
	\end{lemma}
\begin{remark}
	$p$ in  \eqref{eqcostate} is called  the co-state or adjoint state. 
	In addition, by Assumption \ref{A1} we have 
	$$ \|p\|_{2,\Omega_1 \cup \Omega_2} \lesssim \|y\|_{0}+\|y_d\|_{0} . $$
\end{remark}
\begin{remark}\label{rem2.3}
		The variational inequality \eqref{eqprojection} means
	\begin{equation}\label{eqprojection2}
	u=P_{U_{ad}} \left (-\frac{1}{\nu} p \right),
	\end{equation} 
	where $P_{U_{ad}}$ is the $L^2$ projection onto $U_{ad}$. In particular, if $u$ is unconstrained, i.e. $U_{ad}=L^2(\Omega)$,  then  the relation \eqref{eqprojection2} is reduced to 
\begin{equation}\label{eqprojection-cont}
u=-\frac{1}{\nu} p.
\end{equation}
\end{remark}

	\section{Nitsche-XFEM for state and co-state equations}
	\subsection{Extended finite element space}
	
Let $\mathscr{T}_h$ be a shape-regular triangulation of $\Omega$ consisting of open triangles/tetrahedrons with mesh size $h=\max_{T\in \mathscr{T}_h}h_T$, where $h_T$ denotes the diameter of $T\in \mathscr{T}_h$.  We mention that  $\mathscr{T}_h$ is independent of the location of interface.

Define
\begin{align*}
&\mathcal{T}_h^{\Gamma}:=\{T\in \mathscr{T}_h:T\cap \Gamma \neq \emptyset\}. 
\end{align*}
For any   $T\in \mathcal{T}_h^{\Gamma}$,  called  an interface element, we  set $T_m:=T\cap \Omega_m (m=1,2),  \Gamma_T:=\Gamma\cap T$, and denote by $\Gamma_{T,h}$   the straight line/plane segment connecting the intersection between $\Gamma$ and $\partial T$.

For ease of discussion,   we make the following  standard assumptions   on $\mathscr{T}_h$ and   $\Gamma$ (cf. \cite{Hansbo2002An, Schott17sta}).
\begin{itemize}
\item[\bf{(A1)}.]  For $T\in \mathcal{T}_h^{\Gamma}$ and    an edge/face $F\subset \partial T$,    $\Gamma\cap F$ is simply connected. 
\item[\bf{(A2)}.] For $T\in \mathcal{T}_h^{\Gamma}$, there is  a smooth function $\psi$ which maps  $\Gamma_{T,h}$ onto $ \Gamma_T.$
\end{itemize}

\begin{remark}
We note that {\bf{(A1)}} is easily fulfilled for sufficiently fine meshes, and 
{\bf{(A2)}} requires   $\Gamma$ to be piecewise smooth.
\end{remark}

	Denote by $\Theta:=\{P_i:  i=1,2,\cdots,I\}$   the set of  all mesh points of  the triangulation $\mathcal{T}_h$, and by $\Theta_\Gamma:=\Theta\bigcap \bar{\mathcal{T}}_h^{\Gamma}$ the set of  all vertexes of the interface elements. Let $V_h^P$ be the standard linear finite element space with respect to the triangulation $\mathcal{T}_h$  with $\varphi_i\in V_h^P$ denoting the nodal basis function corresponding to the node $P_i$ for $i=1,2,\cdots,I$.  
	
	For any $P_i\in \Theta_\Gamma \bigcap  \Omega_m$ ($m=1,2$),  define  the cut basis function $\widetilde{ \varphi_{i} }$  by
	$$\widetilde{ \varphi_{i} } (x):= \left \{ \begin{array}{ll}
	0, & x\in   \Omega_m,\\
	\varphi_{i}(x), & x\in   \Omega \setminus \Omega_m. 
	\end{array} \right. 
	$$
Then we introduce the cut finite element space
	 $$V_h^\Gamma:=span\{\widetilde{ \varphi_{i} }:  P_i\in \Theta_\Gamma \setminus \Gamma
	 \},$$ 
and  define the extended finite element space  
 $$V_h:=\{v_h \in V_h^P \oplus V_h^\Gamma :v_h |_{\partial\Omega} =0 \}.$$ 
 It is easy to see that   for any $v_h \in V_h$, $v_h|_{\Omega_i}$ ($i=1,2$)  is piecewise linear and continuous, and $v_h$ is discontinuous across the interface $\Gamma$.

%
		
		\subsection{Formulations of Nitsche-XFEM}
		 
%
%
%
		
		
		To describe the Nitsche-XFEM, we first introduce some notations. For each  interface element $T\in \mathcal{T}_h^{\Gamma}$ and $m=1,2$,  we set $$T_m:=T\bigcap{\Omega_m}, \quad k_m:=\frac {|T_m|}{|T|},\ $$
where 	$|T_m|$ and $|T|$ denote the area/volume of $T_m$ and $T$ respectively.  It is evident that
$$k_1+k_2=1.$$
For $\phi\in V^h$, we set 
	$$ \phi_m:=\phi|_{\Omega_m},  \quad \{\phi\}:=(k_1\phi_1+k_2\phi_2)|_{\Gamma}.$$
Introduce the following bilinear form  $a_h(\cdot,\cdot)$: 
 for $w_h,v_h\in V_h$,
\begin{align}
	\label{eq:nfemsc}
a_h(w_h,v_h)&:=( \alpha \nabla w_h,\nabla v_h)_{\Omega_1 \cup \Omega_2} -(\{ \alpha \nabla_\textbf{n} w_h\},[v_h])_{\Gamma}-(\{ \alpha \nabla_\textbf{n} v_h \},[w_h])_{\Gamma}+\lambda ([w_h],[v_h])_{\Gamma},
\end{align}
where  the stabilization parameter $\lambda$ is taken as
\begin{equation}\label{stablization-constant}
\lambda|_T=\widetilde{C}h_T^{-1}\max\{\alpha_1,\alpha_2\},
\end{equation}
with  $\tilde{C}$ a positive constant. 

Then, by following $\cite{Hansbo2002An}$, the Nitsche-XFEMs for the  
 state equation \eqref{eqstate} and the co-state equation \eqref{eqcostate} are respectively given as follows.

Find  $y^h\in V_h$ such that 
	\begin{align} 
	a_h(y^h,v_h)= (u+f,v_h)+( k_2 g,v_{h1} )_{\Gamma}+( k_1 g,v_{h2})_{\Gamma}, 
	\quad \forall v_h \in V_h. 	\label{eqnfem}
	\end{align}
	
Find  $p^h\in V_h$ such that 
\begin{equation} \label{eqaucostate}
		a_h(v_h,p^h)=(y-y_d,v_h), \quad \forall v_h \in V_h.
	\end{equation}
	 
%
\begin{remark}
Note that the  bilinear form  $a_h(\cdot,\cdot)$   corresponds to the symmetric interior penalty Galerkin (SIPG) method \cite{Arnold1982, Wheeler1978}.  

%


\end{remark}

	 \begin{remark}
	 In the stabilization term $\lambda ([w_h],[v_h])_{\Gamma}$ of $a_h(w_h,v_h)$ with $\lambda|_T=\widetilde{C}h_T^{-1}\max\{\alpha_1,\alpha_2\}$, the positive constant $\widetilde{C}$ is required to be ``sufficiently large" to ensure the coercivity of $a_h(\cdot,\cdot)$ (cf. \eqref{coercivity-discreteBilinear}).
	\end{remark}
	
	\begin{remark}
In \cite{Wang2016},   a ``parameter-friendly" DG-XFE scheme was proposed 
for  the following type of interface problem: 
	\begin{equation}\label{eqstrongstate}
\left\{
\begin{array}{rll}
& -\nabla\cdot(\alpha(x)\nabla w)=f & \text{ in }\Omega, \\
& w=0 & \text{ on }\partial\Omega, \\
& [w]=g_D,\ [\alpha\nabla_n w ]=g_{N} & \text{ on }\Gamma,\\
\end{array}
\right.
\end{equation}
where the interface $\Gamma$ is assumed to be  $C^2(\Omega)$-smooth. 
Let $p$ be any given positive integer,  and set 
$$\tilde V_h:=\{ v_h\in H^1(\Omega): \ v_h|_T\in \mathcal{P}_p(T), \forall T\in  \mathcal{T}_h\},$$
 $$W_h:= \chi_1 \tilde V_h+  \chi_2 \tilde V_h,$$
where $ \mathcal{P}_p(T)$ denotes the set of polynomials of degree no more than $p$, and $\chi_m$ is the characteristic function of $\Omega_m$ for $m=1,2$. Then 
the   DG-XFE is formulated as follows: find $w_h\in W_h$ such that
\begin{align*}
		a^*_h(w_h,v_h)=&(f,v_h)+( k_2 {g_N},v_{h1} )_{\Gamma}+( k_1 {g_N},v_{h2})_{\Gamma}-\\ &(g_D,\{\alpha\nabla_n v\})_{\Gamma}+(\lambda_1^* g_D,[v_h])_{\Gamma} +\underset{T\in \mathcal{T}_h^{\Gamma}}  {\sum} (\lambda^*_2 \alpha r_e([g_D]),r_e([v_h])), \quad \forall v_h \in W_h.
\end{align*}
Here 
\begin{align*}
	a^*_h(w_h,v_h)&:=( \alpha \nabla w_h,\nabla v_h)_{\Omega_1 \cup \Omega_2} -(\{ \alpha \nabla_\textbf{n} w_h\},[v_h])_{\Gamma}-(\{ \alpha \nabla_\textbf{n} v_h \},[w_h])_{\Gamma}+ \\ &\lambda^* ([w_h],[v_h])_{\Gamma} +\underset{T\in \mathcal{T}_h^{\Gamma}}  {\sum} (\eta \alpha r_e([w_h]),r_e([v_h])),
\end{align*}
 and, for any  $e=T \cap \Gamma$ with $T\in \mathcal{T}_h^{\Gamma}$,  $ r_e:L^2(e)^d \rightarrow Z_T$ is a lifting operator  given by 
\[ \int_{T} r_e(q) \cdot \alpha z_h= - \int_{e} q \cdot \{ \alpha z_h \}  \quad \forall z_h\in Z_T, \]
where 
\[Z_T=\{ z_h \in L^2(\Omega)^d: z_h|_{T_m} \in P_p(T_m)^d, \ z_{h}|_{\Omega\setminus T}=0 \}.\]
As shown in \cite{Wang2016},   the introduction of the 
  penalization term  based on the lifting operator $r_e$  locally along the interface  guarantees the coercivity  of $a^*_h(\cdot,\cdot)$   as long as the stabilization parameters $\lambda^*|_T \geq h_T^{-1}$ and $\eta\geq 2$. 

	\end{remark}

		Let us introduce a mesh-dependent norm $|||\cdot|||$ on $H^{3/2}(\Omega_1\cup\Omega_2)$: 
	\begin{equation}\label{|||-norm}
	||| v ||| ^2:=\|\nabla v\|^2_{0,{\Omega_1 \cup \Omega_2}}+\|\{ \nabla_\textbf{n} v \}\|^2_{-1/2,h,\Gamma}+\|[v] \|^2_{1/2,h,\Gamma}, \quad \forall v\in H^{3/2}(\Omega_1\cup\Omega_2),
	\end{equation}
where	$$\|v\|^2_{1/2,h,\Gamma}:= \underset{T \in \mathcal{T}_h^{\Gamma}}  {\sum} h_T^{-1} \|v\|^2_{0,\Gamma_T}, \quad \|v\|^2_{-1/2,h,\Gamma}:= \underset{T \in \mathcal{T}_h^{\Gamma}}  {\sum} h_T \|v\|^2_{0,\Gamma_T}.$$
It is easy to see that $|||\cdot|||$ is a norm on $V^h$ with
\begin{equation}\label{poincare-inequality}
||v_h||_{0,\Omega}\lesssim |v_h|_{1,\Omega_1\cup\Omega_2} \leq |||v_h|||, \quad \forall v_h\in V^h.
\end{equation}

Under the  assumptions {\bf (A1)}-{\bf (A2)},     the following  boundedness and coerciveness results hold (cf. \cite{Arnold1982,Hansbo2002An}):
\begin{equation} \label{boundness-discreteBilinear}
a_h(w,v)\lesssim |||w|||\ |||v|||, \quad \forall w, v\in H^{3/2}(\Omega_1\cup\Omega_2),
\end{equation}
and
\begin{equation}\label{coercivity-discreteBilinear}
a_h(v_h,w_h)\gtrsim |||v_h|||^2, \quad \forall v_h\in V^h
\end{equation}
if $\tilde{C}$ in \eqref{stablization-constant} is   sufficiently large. Hence, the discrete problems \eqref{eqnfem} and  \eqref{eqaucostate}  admit unique solutions $y^h\in V_h$ and $p^h\in V_h$, respectively. In addition, from \cite{Hansbo2002An, Yang2018} we have the following error estimates.

%

\begin{lemma} $ \cite{Hansbo2002An}$
	\label{lepriorestimate}
	Let $y,p \in H^1_0(\Omega)\cap H^{2}(\Omega_1 \cup \Omega_2)$ be the solutions to the weak problems \eqref{eqstate} and \eqref{eqcostate}, respectively. Then it holds
	$$||| y-y^h |||  \lesssim h\|y\|_{2,\Omega_1 \cup \Omega_2},\quad ||| p-p^h |||  \lesssim h\|p\|_{2,\Omega_1 \cup \Omega_2}.$$ 
	$$\|y-y^h\|_0 \lesssim h^2\|y\|_{2,\Omega_1 \cup \Omega_2}, \quad \|p-p^h\|_0 \lesssim h^2\|p\|_{2,\Omega_1 \cup \Omega_2}.$$
\end{lemma}

\begin{lemma} $\cite{Yang2018}$ \label{lepriorestimate2}
Let $y,p \in H^1_0(\Omega)\cap H^{3/2}(\Omega_1 \cup \Omega_2)$ be the solutions to the weak problems \eqref{eqstate} and \eqref{eqcostate}, respectively.  Then it holds
$$
|||y-y^h|||\lesssim h^\frac{1}{2} \|y\|_{\frac{3}{2},\Omega_1\cup\Omega_2},\quad
|||p-p^h|||\lesssim h^\frac{1}{2} \|p\|_{\frac{3}{2},\Omega_1\cup\Omega_2},$$
$$\|y-y^h\|_{0,\Omega} \lesssim h \|y\|_{\frac{3}{2},\Omega_1\cup\Omega_2},\quad
\|p-p^h\|_{0,\Omega} \lesssim h \|p\|_{\frac{3}{2},\Omega_1\cup\Omega_2}.$$
\end{lemma}

\section{Discretization of optimal control problem}
\subsection{Discrete optimality conditions}

By  following the variational discretization concept in \cite{Hinze05var},  the optimal control problem \eqref{eqobjective}-\eqref{eqconstraint}
is approximated by the following discrete optimal control problem: 
%
\begin{equation} \label{eqdobjective}
\min\limits_{(y_h,u)\in V_h \times U_{ad}} J_h(y_h,u)=\frac{1}{2} \int_{\Omega} (y_h-y_d)^2 dx+\frac{\nu}{2}\int_{\Omega} u^2 dx
\end{equation} 
with 
\begin{equation}\label{eqdstate}
a_h(y_h,v_h)= (u+f,v_h)+( k_2 g,v_{h1} )_{\Gamma}+( k_1 g,v_{h2})_{\Gamma}, \quad \forall v_h \in V_h.
\end{equation}
Similar to Lemma \ref{leoptimalcondition},   the following lemma holds.
\begin{lemma}
	The discrete optimal control problem   \eqref{eqdobjective}-\eqref{eqdstate} has a unique solution, and the solution  $(y_h,p_h,u_h)\in V_h\times V_h\times U_{ad}$ satisfies the following optimality conditions:
\begin{align}
& a(y_h,v_h)=(u_h+f,v_h)+( k_2 g,v_{h1} )_{\Gamma}+( k_1 g,v_{h2})_{\Gamma}, \quad \forall v_h\in V_h,\label{eqdstate2}\\
& a(v_h,p_h)=(y_h-y_d,v_h),\quad \forall v_h\in V_h,\label{eqdcostate}\\
& (p_h+a u_h,v-u_h) \geq 0,\quad \forall v\in U_{ad}.\label{eqdprojection}
\end{align}
\end{lemma}
\begin{remark}\label{rem-dis-ph}
Notice that    the discrete optimal control $u_h\in U_{ad}$   is not directly discretized in the objective functional \eqref{eqdobjective}, since $U_{ad}$ is infinite dimensional. However,   the variational inequality \eqref{eqdprojection} means that   $u_h$ is implicitly discretized through the discrete co-state $p_h $ and   the projection $P_{U_{ad}}$  (cf. \eqref{eqprojection2}) with
\begin{equation}\label{eqprojection3}
u_h=P_{U_{ad}}  \left (-\frac{1}{\nu} p_h \right).
\end{equation}
Moreover, if $u_0$ and $u_1$ are well-defined at any  $x\in \Omega$, then \eqref{eqprojection3} is equivalent to
\begin{equation}\label{eqprojection4}
u_h=\min \left  \{u_1,\max \left  \{u_0, -\frac{1}{\nu} p_h \right \} \right \}.
\end{equation}
In particular, if  $U_{ad}=L^2(\Omega)$, then we have
\begin{equation}\label{eqprojection-dis}
u_h=-\frac{1}{\nu} p_h.
\end{equation}
\end{remark}

\subsection{Error estimates}
In this subsection, we   first  show that the errors between $(y,p,u)$ and $(y_h,p_h,u_h)$,   the solutions to the continuous  optimal control problem  \eqref{eqstate}-\eqref{eqprojection} and to  the discrete optimal control problem \eqref{eqdstate2}-\eqref{eqdprojection} respectively,  can be bounded from above by the errors  between $(y,p)$ and $(y^h,p^h)$. Here 
	we recall that $y^h\in V_h$ and $p^h\in V_h$ are the solutions to the Nitsche-XFE schemes \eqref{eqnfem} and \eqref{eqaucostate}, respectively. 
	
%

\begin{theorem}\label{th1}
	Let $(y,p,u)\in H_0^1(\Omega)\times H_0^1(\Omega)\times U_{ad}$ and $(y_h,p_h,u_h)\in V_h\times V_h\times U_{ad}$ be the solutions to the continuous  problem  \eqref{eqstate}-\eqref{eqprojection} and the discrete   problem \eqref{eqdstate2}-\eqref{eqdprojection}, respectively. 
	Then we have 
	\begin{eqnarray}
	\nu^{\frac12} \|u-u_h\|_0+\|y-y_h\|_0 &\lesssim& \|y-y^h\|_0+{\nu^{-\frac12}}\|p-p^h\|_0, \label{t1} \\
	\|p-p_h\|_{0} &\lesssim& \|p-p^h\|_{0} + \|y-y_h\|_{0}, \label{t2} \\
	|||y-y_h||| &\lesssim& |||y-y^h|||+\|u-u_h\|_{0},  \label{t3} \\
	|||p-p_h|||&\lesssim& |||p-p^h|||+\|y-y_h\|_{0}.\label{t4}
	\end{eqnarray}
\end{theorem}
\begin{proof}
	First, by \eqref{eqdstate2}-\eqref{eqdcostate} and \eqref{eqnfem}-\eqref{eqaucostate}
	we have
	\begin{align}
		a_h(y_h-y^h,v_h)&=(u_h-u,v_h), \quad \forall v_h \in V_h, \label{eqt1} \\
		a_h(v_h,p_h-p^h)&=(y_h-y,v_h), \quad \forall v_h \in V_h, \label{eqt2}
	\end{align}
	which yield
	\begin{align}
    (y_h-y,y_h-y^h)&=a_h(y_h-y^h,p_h-p^h) 
 	=(u_h-u,p_h-p^h)\label{eqt4}.
	\end{align}
 Take $v=u_h$ in \eqref{eqprojection} and $v=u$ in \eqref{eqdprojection}, we get
	\begin{align*}
	(\nu u+p,u_h-u) &\geq 0, \\
 	(\nu u_h+p_h,u-u_h) &\geq 0.  
	\end{align*}
Adding together these two inequalities implies $$(\nu(u-u_h)+p-p_h,u_h-u) \geq 0,$$
which, together with   \eqref{eqt4}, leads to 
	\begin{align*}
		 \nu\|u-u_h\|^2_0 & \leq (u_h-u,p-p_h) \\
		 &=(u_h-u,p-p^h)+(u_h-u,p^h-p_h) 	\\
		  &=(u_h-u,p-p^h)+(y_h-y,y^h-y_h)  \\
		  &\leq \frac{1}{2}(\nu \|u_h-u\|^2_0+\frac{1}{\nu}\|p-p^h\|^2_0)+(y_h-y,y^h-y_h) \\	
		  &\leq \frac{1}{2}(\nu \|u_h-u\|^2_0+\frac{1}{\nu}\|p-p^h\|^2_0) -\frac{1}{2}\|y-y_h\|_0^2+\frac{1}{2}\|y-y^h\|_0^2.
	\end{align*}
	Consequently,  \eqref{t1} holds. 
	
	Second, let us show \eqref{t2} and \eqref{t3}. From \eqref{poincare-inequality}, \eqref{coercivity-discreteBilinear}
	  and \eqref{eqt2}, we have
\begin{align*}
	|| p_h-p^h||^2_{0}&\lesssim |||p_h-p^h|||^2\\
	&\lesssim a_h(p_h-p^h,p_h-p^h)=(y_h-y, p_h-p^h)\\
	&\lesssim ||y_h-y||_{0}|| p_h-p^h||_{0}\\
	&\lesssim ||y_h-y||_{0}||| p_h-p^h|||,
\end{align*}
which, together with the triangle inequality, yields
\begin{align*}
\|p-p_h\|_0 & \leq \|p-p^h\|_0+\|p^h-p_h\|_0   \lesssim \|p-p^h\|_0+\|y_h-y\|_0 ,\\
|||p-p_h||| & \leq |||p-p^h|||+|||p^h-p_h|||   \lesssim |||p-p^h|||+\|y_h-y\|_0,\\
\end{align*}
i.e.   \eqref{t2} and  \eqref{t3} hold.

Similarly, \eqref{t3}   follows from    \eqref{poincare-inequality}, \eqref{coercivity-discreteBilinear} and \eqref{eqt1}.
%
\end{proof}

 Based on Theorem \ref{th1},  Lemmas \ref{lepriorestimate}-\ref{lepriorestimate2},  and Remarks \ref{rem2.3} and \ref{rem-dis-ph},  we immediately have the following main results of error estimation.

\begin{theorem}\label{th2}
	Let $(y,p,u)\in \left(H_0^1(\Omega)\cap H^s (\Omega_1 \cup \Omega_2)\right)\times \left(H_0^1(\Omega)\cap H^s (\Omega_1 \cup \Omega_2)\right)\times U_{ad}$ ($s=2,3/2$) and $(y_h,p_h,u_h)\in V_h\times V_h\times U_{ad}$ be  the solutions to the continuous problem (\ref{eqstate})-(\ref{eqprojection}) and the discrete problem (\ref{eqdstate2})-(\ref{eqdprojection}), respectively. Then we have, for $s=2$, 
	\begin{eqnarray}
	\|u-u_h\|_{0}+\|y-y_h\|_{0} +\|p-p_h\|_{0} &\lesssim& h^2 (\|y\|_{2,\Omega_1\cup \Omega_2}+\|p\|_{2,\Omega_1\cup \Omega_2}),\label{u0}\\
	|||y-y_h||| + |||p-p_h|||&\lesssim& h (\|y\|_{2,\Omega_1\cup \Omega_2}+\|p\|_{2,\Omega_1\cup \Omega_2}), \label{y00}
	\end{eqnarray}
	and for $s=3/2$,
\begin{eqnarray}
	\|u-u_h\|_{0}+\|y-y_h\|_{0} +\|p-p_h\|_{0} &\lesssim& h  (\|y\|_{\frac32,\Omega_1\cup \Omega_2}+\|p\|_{\frac32,\Omega_1\cup \Omega_2}),\label{u00}\\
	|||y-y_h||| + |||p-p_h|||&\lesssim& h^\frac12 (\|y\|_{\frac32,\Omega_1\cup \Omega_2}+\|p\|_{\frac32,\Omega_1\cup \Omega_2}). \label{y000}
	\end{eqnarray}
	In particular, if $u$ is unconstrained, i.e. $U_{ad}=L^2(\Omega)$, then we further have
		\begin{eqnarray}
	|||u-u_h||| &\lesssim& h^{s-1} (\|y\|_{s,\Omega_1\cup \Omega_2}+\|p\|_{s,\Omega_1\cup \Omega_2}), \quad s=2, 3/2.\label{u000}
\end{eqnarray}

	\begin{remark}
	 In view of the definition of   $ |||\cdot||| $ in \eqref{|||-norm}, the estimates \eqref{y00} and \eqref{y000}  indicate		
	 \begin{eqnarray}
		|y-y_h|_1 + |p-p_h|_1&\lesssim& h^{s-1} (\|y\|_{s,\Omega_1\cup \Omega_2}+\|p\|_{s,\Omega_1\cup \Omega_2}),  \quad s=2, 3/2,
	\end{eqnarray}
	and    \eqref{u000} indicates	
\begin{eqnarray}
|u-u_h|_1 &\lesssim& h^{s-1} (\|y\|_{s,\Omega_1\cup \Omega_2}+\|p\|_{s,\Omega_1\cup \Omega_2}),  \quad s=2, 3/2,
	\end{eqnarray}
	where  $|\cdot|_1:=|\cdot|_{1,\Omega_1\cup\Omega_2}.$
	\end{remark}

%
\end{theorem}

	\section{Numerical results}
	
	We shall provide several 2D numerical examples to verify the performance of the Nitsche-XFEM. Note that the optimal control problem
\eqref{eqobjective}-\eqref{eqstrongstate0} without the constraint \eqref{eqconstraint} is a linear problem, the resultant discrete linear system is easy to solve. However, for the constrained optimal control problem \eqref{eqobjective}-\eqref{eqconstraint}, the corresponding discrete optimal control problem \eqref{eqdobjective}-\eqref{eqdstate} or its equivalent  problem \eqref{eqdstate2}-\eqref{eqdprojection} is a nonlinear system, which we shall apply   the following fixed-point iteration algorithm to solve. 



\paragraph{Algorithm}
\noindent\textbf{Fixed-point iteration}
\begin{enumerate}
	\item Initialize $u^i_h=u^0$; 
	\item Compute $y^i_h\in V_h$ by $a_h(y_h^i,v_h)=(u_h^i,v_h)+(f,v_h)+(k_2g,v_{1,h})_{\Gamma_h}+(k_1g,v_{2,h})_{\Gamma_h}
	,\forall v_h\in V_h$; \label{computeyhi}
	\item Compute $p_h^i\in V_h$ by $a_h(v_h,p_h^i)=(y_h^i-y_d,v_h),\forall v_h\in V_h$;
	\item Set $u_h^{i+1}=\min \left  \{u_1,\max \left  \{u_0, -\frac{1}{\nu} p_h^i \right \} \right \};$
	\item if $|u_h^{i+1}-u_h^i|<\text{Tol}$ or $i+1>\text{MaxIte}$, then output $u_h=u_h^{i+1}$, else $i=i+1$, and go back to Step \ref{computeyhi}.
\end{enumerate}
Here $u^0$ is an initial value, Tol is the tolerance, and MaxIte is the maximal iteration number. 
 Theoretically, this algorithm is convergent when the regularity parameter $\nu$ is large enough (cf. \cite{M2008}). 

In each example,  we choose $\Omega $ to be a square, and use $N\times N$  uniform  meshes with  $2N^2$ triangular elements. 

  \begin{example}		
 		{Segment interface: a case without control constraints}
 		\end{example}	
		
		Consider the    optimal control problem
\eqref{eqobjective}-\eqref{eqstrongstate0} without the constraint \eqref{eqconstraint}. Set  the regulation parameter $\nu=0.01$, the domain $\Omega:=[0,1]\times[0,1]$ (cf. Figure \ref{fgin1}), the  interface
$$\Gamma:=\{(x_1,x_2):x_2=kx_1+b\}\cap \Omega$$ 
with $ k=-\sqrt{3}/3, b=(6+\sqrt{6}-2\sqrt{3})/6$, and
 $$\Omega_1:=\{(x_1,x_2):kx_1+b-x_2>0 \} \cap \Omega , \quad \Omega_2:=\{(x_1,x_2):kx_1+b-x_2<0 \} \cap \Omega.$$
  Take the coefficients $\alpha|_{\Omega_1}=\alpha_1:=1, \ \alpha|_{\Omega_2}=\alpha_2:=100$,
and the control  space $U_{ad}:=L^2(\Omega)$.
  Let $y_d, f, g$ be  such that the optimal triple $(y,p,u)$ of  \eqref{eqstate}-\eqref{eqprojection} is of the  the following form:
		\begin{eqnarray*}
	y(x_1,x_2)&=&
	\begin{cases}
	\frac{(x_2-kx_1-b)cos(x_1x_2)}{2\alpha_1}+(x_2-kx_1-b)^3, & \mbox{ in $ \Omega_1$, } \\
	\frac{(x_2-kx_1-b)cos(x_1x_2)}{2\alpha_2}, &\mbox{ in $ \Omega_2$, }
	\end{cases}\\
	u(x_1,x_2)&=&
	\begin{cases}
	\alpha_2(x_2-kx_1-b)x_1(x_1-1)x_2(x_2-1)sin(x_1x_2),& \mbox{ in $ \Omega_1$, } \\
	\alpha_1(x_2-kx_1-b)x_1(x_1-1)x_2(x_2-1)sin(x_1x_2),&\mbox{ in $ \Omega_2$, }
	\end{cases}\\
	p(x_1,x_2)&=&-\nu u(x_1,x_2).
	\end{eqnarray*}
	
	We compute the discrete schemes  \eqref{eqdstate2}-\eqref{eqdprojection} with the stabilization parameter $\widetilde{C}=10$  (cf. \eqref{stablization-constant}). 
	Tables \ref{Tab-1-eg1}-\ref{Tab-2-eg1} give  numerical  results of  the relative errors between $(y_h,p_h,u_h)$ and $(y,p,u)$ in the $L^2-$norm and the $H^1$-seminorm, respectively. We can see that the Nitsche-XFEM  yields optimal convergence orders, i.e. second order rates of convergence for $|y-y_h|_0$, $|p-p_h|_0$ and $|u-u_h|_0$, and   first order rates of convergence for $|y-y_h|_1$, $|p-p_h|_1$ and $|u-u_h|_1$. This is consistent with our theoretical results in Theorem \ref{th1}.

		\begin{figure}[hbtp] 
		\centering
		\includegraphics[width=8cm]{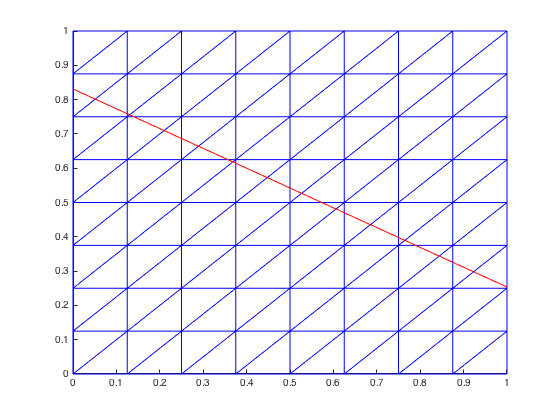}
		\caption{ Segment  interface and $8\times 8$ mesh  for Example 5.1}
		\label{fgin1}
	\end{figure}
	
	\begin{table}[!hbtp]
		\centering
				\caption{Relative errors of Nitsche-XFEM in $L^2$-norm for Example 5.1.}\label{Tab-1-eg1}
\begin{tabular}{|c|c|c|c|c|c|c|}
			\hline
			N & $\frac{\|u-u_h\|_0}{\|u\|_0}$ & order & $\frac{\|y-y_h\|_0}{\|y\|_0}$ & order & $\frac{\|p-p_h\|_0}{\|p\|_0}$& order  \\
			\hline
			16 &3.9941e-02&      &8.7667e-03&      &3.9941e-02&      \\
			32 &9.6399e-03&2.1&2.1955e-04&2.0&9.6399e-03&2.1\\
			64 &2.3780e-03&2.0&5.5005e-04&2.0&2.3780e-03&2.0\\
			128&5.9194e-04&2.0&1.3834e-05&2.0&5.9195e-04&2.0\\
			256&1.4794e-04&2.0&3.5203e-06&2.0&1.4794e-04&2.0\\
			\hline
		\end{tabular}
	\end{table}
	\begin{table}[!hbtp]
		\centering
		\caption{Relative errors of Nitsche-XFEM in $H^1$-seminorm for Example 5.1.}\label{Tab-2-eg1}
	\begin{tabular}{|c|c|c|c|c|c|c|}
			\hline
			N & $\frac{|u-u_h|_1}{|u|_1}$ & order & $\frac{|y-y_h|_1}{|y|_1}$ & order & $\frac{|p-p_h|_1}{|p|_1}$& order  \\
			\hline
			16 &2.0695e-01&       &1.0180e-01&       &2.0695e-01&      \\
			32 &1.0404e-01&1.0&5.0958e-02&1.0&1.0404e-01&1.0\\
			64 &5.2064e-02&1.0&2.5486e-02&1.0&5.2064e-02&1.0\\
			128&2.6035e-02&1.0&1.2744e-02&1.0&2.6035e-02&1.0\\
			256&1.3017e-02&1.0 &6.3722e-03&1.0&1.3017e-02&1.0 \\
			\hline
		\end{tabular}
		\end{table}

%


\begin{example} {Circle interface: a case without control constraints}
\end{example}
	 This example is from  \cite{Zhang2015Immersed},  where  it was used to  test the performance of an IFEM. 
In  the    optimal control problem
\eqref{eqobjective}-\eqref{eqstrongstate0},  take $\nu=0.01$ and $\Omega=[-1,1]\times[-1,1]$.	The interface $\Gamma$ is a circle  centered at $(0,0)$ with radius $r= \frac12$. Set
	 $$\Omega_1:=\{(x_1,x_2):x_1^2+x_2^2< r^2 \}, \quad \Omega_2:=\{(x_1,x_2):x_1^2+x_2^2>r^2 \}\cap \Omega,$$ 
 $\alpha|_{\Omega_1}=\alpha_1:=1, \ \alpha|_{\Omega_2}=\alpha_2:=10$,
and  $U_{ad}:=L^2(\Omega)$.   Let $y_d, f, g$ be  such that the optimal triple $(y,p,u)$ of  \eqref{eqstate}-\eqref{eqprojection} is of the  the following form:
	\begin{eqnarray*}
	y(x_1,x_2)&=&
	\begin{cases}
	\frac{(x_1^2+x_2^2)^{\frac32}}{\alpha_1}, & \mbox{ in $ \Omega_1$ }, \\
	\frac{(x_1^2+x_2^2)^{\frac32}}{\alpha_2} +( \frac{1}{\alpha_1}-\frac{1}{\alpha_2})r^3, &\mbox{ in $ \Omega_2$ },
	\end{cases}\\
		u(x_1,x_2)&=&
	\begin{cases}
	\frac{5(x_1^2+x_2^2-r^2)(x_1^2-1)(x_2^2-1)}{\alpha_1},& \mbox{ in $ \Omega_1$ }, \\
	\frac{5(x_1^2+x_2^2-r^2)(x_1^2-1)(x_2^2-1)}{\alpha_2},  &\mbox{ in $ \Omega_2$ },
	\end{cases}\\
		p(x_1,x_2)&=&-\nu u(x_1,x_2).
	\end{eqnarray*}
	Notice that $g=0$ in this example. 
	
		  In the   schemes  \eqref{eqdstate2}-\eqref{eqdprojection} we take  the stabilization parameter $\widetilde{C}=1000$, and  use the polygonal line $\Gamma_h=\bigcup\limits_{T\in \mathcal T_h} \Gamma_{T,h}$   to replace the exact interface $\Gamma$.
		  Tables \ref{tab5.3}-\ref{tab5.4} give some numerical  results of  the  errors  in the $L^2-$norm and the $H^1$-seminorm, respectively. For comparison we also list the results from \cite{Zhang2015Immersed} obtained by the classical IFEM. We can see that  the Nitsche-XFEM  yields optimal convergence orders for all the $L^2$ and $H^1$ errors. In particular, the convergence rates of Nitsche-XFEM are always full when the mesh is refined, while    the rates of  IFEM may deteriorate, e.g.  the  rate of $|u-u_h|_1$ deteriorates from  $1.01$ at the $32\times32$ mesh  to $0.91$ at the $256\times256$ mesh. In fact, such phenomenon of accuracy deterioration for IFEM has been observed  in \cite{Lins15par} for elliptic interface problems.

\begin{table}[!hbtp]
	\centering
	\caption{$L^2$ errors of Nitsche-XFEM (abbr. NXFEM) and IFEM for Example 5.2.}\label{tab5.3}
	\begin{tabular}{|c|c|c|c|c|c|c|c|}
		\hline
		Method&N & $\|u-u_h\|_0$ & order & $\|y-y_h\|_0$ & order & $\|p-p_h\|_0$& order  \\
		\hline
		&16 &1.1316e-02&       &4.4535e-03&      &1.1316e-04&      \\
		&32 &3.0688e-03&1.88&1.1883e-03&1.91&3.0688e-05&1.88\\
		NXFEM&64 &7.5979e-04&2.01&3.1686e-04&1.91&7.5979e-06&2.01\\
		&128&1.8516e-04&2.04&7.6393e-05&2.05&1.8516e-06&2.04\\
		&256&4.2966e-05&2.11&1.8584e-05&2.04&4.2966e-07&2.11\\
		\hline
		&	16 &1.1889e-02&       &4.6400e-03&      &1.1889e-04&      \\
		&32 &3.1406e-03&1.92&1.2288e-03&1.91&3.1406e-05&1.92\\
		IFEM\cite{Zhang2015Immersed}&64 &7.0663e-04&2.15&3.1438e-04&1.96&7.0663e-06&2.15\\
		&128&1.6334e-04&2.11&8.1934e-05&1.93&1.6334e-06&2.11\\
		&256&3.5894e-05&2.18&2.1650e-05&1.92&3.5894e-07&2.18\\
		\hline
	\end{tabular}
\end{table}

\begin{table}[!hbtp]
	\centering
	\caption{$H^1$ errors of Nitsche-XFEM  and IFEM for Example 5.2.}\label{tab5.4}
	\begin{tabular}{|c|c|c|c|c|c|c|c|}
		\hline
		Method&N & $|u-u_h|_1$ & order & $|y-y_h|_1$ & order & $|p-p_h|_1$& order  \\
		&16 &1.1407e-01&       &1.1311e-01&      &1.1401e-03&      \\
		&32 &5.7015e-02&1.00&5.8796e-02&0.94&5.6926e-04&1.00\\
		NXFEM&64 &2.7869e-02&1.03&2.9448e-02&1.00&2.7932e-04&1.03\\
		&128&1.3830e-02&1.01&1.4800e-02&0.99&1.3852e-04&1.01\\
		&256&6.8465e-03&1.01&7.3659e-03&1.00&6.8465e-05&1.01\\
		\hline
		&		16 &1.0665e-01&       &1.0778e-01&      &1.0665e-03&      \\
		&32 &5.2602e-02&1.01&5.5660e-02&0.95&5.2602e-04&1.01\\
		IFEM\cite{Zhang2015Immersed}&64 &2.7054e-02&0.95&2.9084e-02&0.93&2.7054e-04&0.95\\
		&128&1.4028e-02&0.94&1.5047e-02&0.95&1.4028e-04&0.94\\
		&256&7.4170e-03&0.91&7.9081e-03&0.92&7.4170e-05&0.91\\
		\hline
	\end{tabular}
\end{table}



\begin{example} {Circle Interface: a case with  control constraints } 
\end{example}
	Consider the    optimal control problem
\eqref{eqobjective}-\eqref{eqconstraint} with   $\nu=1$ and $\Omega=[-1,1]\times[-1,1]$ (cf. Figure \ref{figu2}). The
	 interface $\Gamma$ is a circle centered at $(0,0)$ with radius $r= \frac{\sqrt{3}}{4}$. Set 
	$$\Omega_1:=\{(x_1,x_2):x_1^2+x_2^2< r^2 \}, \quad \Omega_2:=\{(x_1,x_2):x_1^2+x_2^2>r^2 \}\cap \Omega,$$
$\alpha|_{\Omega_1}=\alpha_1:=1, \ \alpha|_{\Omega_2}=\alpha_2:=1000$,
and   $U_{ad}:=\{u\in L^2(\Omega):-\frac12 \leq u \leq \frac12 \ \text{ a.e in} \ \Omega \}$.
Let $y_d, f, g$ be  such that the optimal triple $(y,p,u)$ of  \eqref{eqstate}-\eqref{eqprojection} is of the  the following form:
	 \begin{eqnarray*}
	 y(x_1,x_2)&=&
	 \begin{cases}
	 \frac{(x_1^2+x_2^2)^{\frac32}}{\alpha_1}-10(x_1^2+x_2^2-r^2)sin(x_1x_2), & \mbox{ in $ \Omega_1$ } \\
	 \frac{(x_1^2+x_2^2)^{\frac32}}{\alpha_2} +( \frac{1}{\alpha_1}-\frac{1}{\alpha_2})r^3, &\mbox{ in $ \Omega_2$, }
	 \end{cases}\\
%
	 u(x_1,x_2)&=& \min \left \{\frac12,\max \left \{-\frac12,\varphi (x_1,x_2) \right \}  \right \},\\
%
	 p(x_1,x_2)&=&-\nu\varphi(x_1,x_2),
	 \end{eqnarray*}
	where    \begin{eqnarray*}
	 \varphi(x_1,x_2):=
	 \begin{cases}
	 \frac{5(x_1^2+x_2^2-r^2)(x_1^2-1)(x_2^2-1)}{\alpha_1}, & \mbox{ in $ \Omega_1$ } \\
	 \frac{5(x_1^2+x_2^2-r^2)(x_1^2-1)(x_2^2-1)}{\alpha_2},  &\mbox{ in $ \Omega_2$ }.
	 \end{cases}
	 \end{eqnarray*} 
	  In the   schemes  \eqref{eqdstate2}-\eqref{eqdprojection} we take  the stabilization parameter $\widetilde{C}=5$. 
	Tables \ref{Tab-1-eg2}-\ref{Tab-2-eg2} give some numerical  results of  the relative errors  in the $L^2-$norm and the $H^1$-seminorm, respectively. We can see that the NXFEM  yields 
	second order rates of convergence for $|y-y_h|_0$, $|p-p_h|_0$ and $|u-u_h|_0$, and   first order rates of convergence for $|y-y_h|_1$ and $|p-p_h|_1$. 
	This is consistent with   Theorem \ref{th1}.

	In Figures \ref{figcontrol}-\ref{figstate}  we show the exact solutions of the control $u$ and state $p$,  and  the Nitsche-XFEM solutions $u_h$  and $p_h$  at the 32 $\times$ 32   mesh.   Figure \ref{figactiveset} demonstrates  the boundaries of the exact   and the computed active sets. 
	We can see that all the numerical approximations match the exact solutions well.  

	 	\begin{figure}[htbp]
	
	\centering
	\includegraphics[width=7cm]{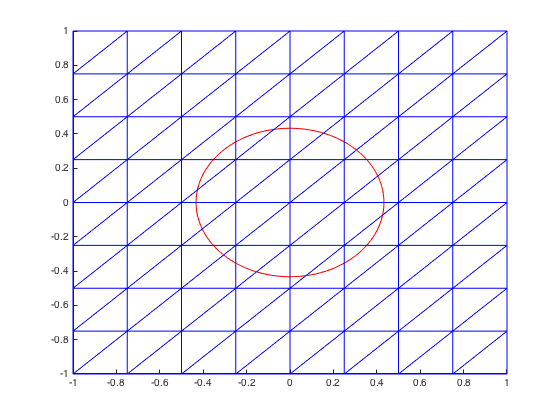}
	\caption{Circle interface and $8\times 8$ mesh for Example 5.3.}
	\label{figu2}
	\label{meshcircle}
\end{figure}

	 \begin{table}[!hbtp]
	 	\centering
				\caption{Relative errors of Nitsche-XFEM in $L^2$-norm for Example 5.3.}\label{Tab-1-eg2}
\begin{tabular}{|c|c|c|c|c|c|c|}
	 		\hline
	 		N & $\frac{\|u-u_h\|_0}{\|u\|_0}$ & order & $\frac{\|y-y_h\|_0}{\|y\|_0}$ & order & $\frac{\|p-p_h\|_0}{\|p\|_0}$& order  \\
	 		\hline
	 		16 &4.4640e-02&       &6.7792e-02&      &5.9076e-02&      \\
	 		\hline
	 		32 &1.7953e-02&1.3&2.3134e-02&1.6&1.8254e-02&1.7\\
	 		\hline
	 		64 &3.9458e-03&2.2&5.7710e-03&2.0&3.9865e-03&2.2\\
	 		\hline
	 		128&7.7806e-04&2.3&1.3023e-03&2.2&8.2130e-04&2.3\\
	 		\hline
	 		256&1.2751e-04&2.6&2.0961e-04&2.6&1.5615e-04&2.4\\
	 		\hline
	 	\end{tabular}
	 \end{table}
	 
	 \begin{table}[!hbtp]
	 	\centering
		\caption{Relative errors of Nitsche-XFEM in $H^1$-seminorm for Example 5.3.}\label{Tab-2-eg2}
	\begin{tabular}{|c|c|c|c|c|}
	 		\hline
	 		N & $\frac{|y-y_h|_1}{|y|_1}$ & order & $\frac{|p-p_h|_1}{|p|_1}$& order  \\
	 		\hline
	 		16 &5.0048e-01&       &2.0831e-01&\\
	 		\hline
	 		32 &2.4468e-01&1.0&1.0421e-01&1.0\\
	 		\hline
	 		64 &1.1515e-01&1.1&4.9146e-02&1.1\\
	 		\hline
	 		128&5.7116e-02&1.0&2.4365e-02&1.0\\
	 		\hline
	 		256&2.6058e-02&1.1&1.1514e-02&1.1\\
	 		\hline
	 	\end{tabular}
		 \end{table}
 
 \begin{figure}[!hbtp]
 	\centering
 	
 	\begin{minipage}{7cm}
 		\includegraphics[width=7cm]{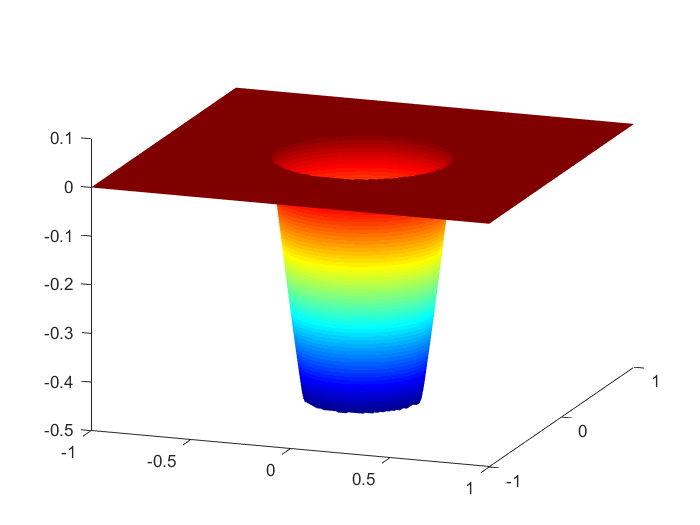}
 	\end{minipage}
 	\begin{minipage}{7cm}
 		\includegraphics[width=7cm]{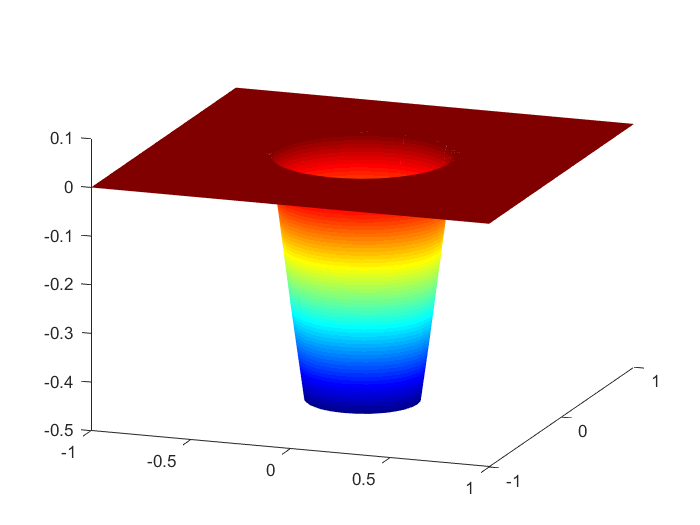}
 	\end{minipage}
 	\caption{ The exact(left) control $u$ and discrete control $u_h$ (right)   for Example 5.3.}
 	\label{figcontrol}
 \end{figure}

 \begin{figure}[!hbtp]
	\centering
	
	\begin{minipage}{7cm}
		\includegraphics[width=7cm]{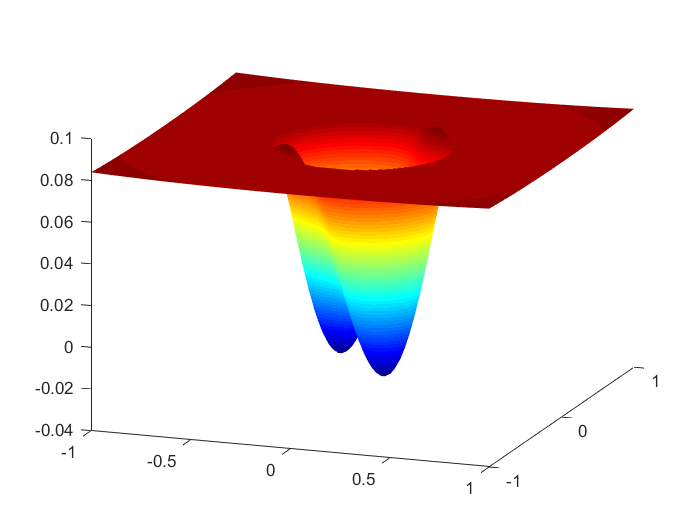}
	\end{minipage}
	\begin{minipage}{7cm}
		\includegraphics[width=7cm]{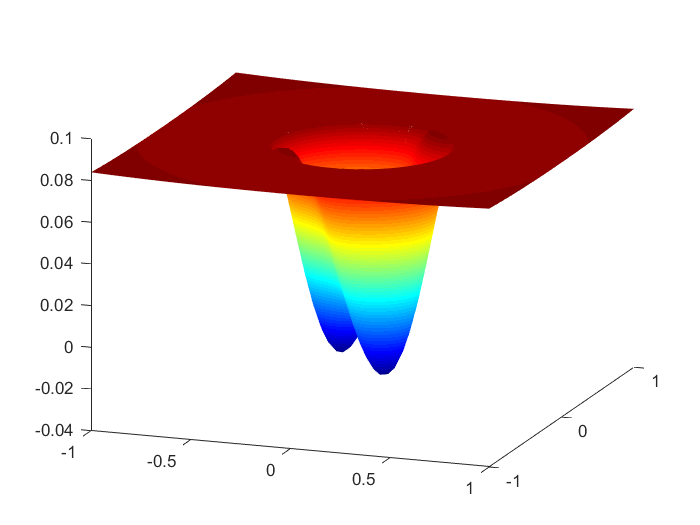}
	\end{minipage}
	\caption{ The exact state $y$ (left) and  discrete state $y_h$ (right)   for Example 5.3.}
	\label{figstate}
\end{figure}

 \begin{figure}[!hbtp]
	\centering
	\includegraphics[width=10cm]{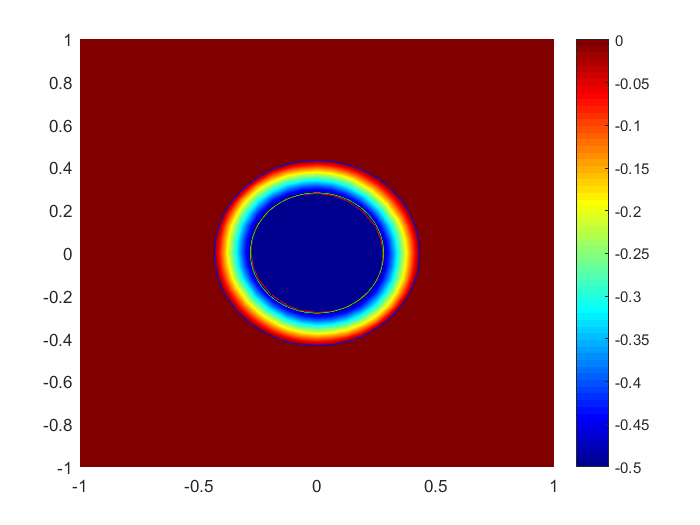}
	\caption{The discrete control $u_h$ for Example 5.3:  The green and red lines are boundaries of the exact and computed active sets, respectively, and the blue line is the interface $\Gamma_h$.}
		\label{figactiveset}
		\end{figure}

\section{Conclusion}
In this paper, the Nitsche eXtended finite element method as well as the variational
discretization concept has been applied to discretize  the distributed optimal control problems of elliptic interface equations. This method does not require interface-fitted meshes, and is  suitable for generic interface conditions. Error analysis and numerical results have demonstrated its optimal convergence  and good performance.  
%
%
%


\end{document}